\documentclass[a4paper,10pt]{scrartcl}

\usepackage[fleqn]{amsmath}
\usepackage{amssymb}
\usepackage{amsthm}
\usepackage{stmaryrd}
\usepackage{bbm}
\usepackage{hyperref}

\DeclareMathOperator{\Var}{Var}
\DeclareMathOperator{\dom}{dom}
\DeclareMathOperator{\supp}{supp}

\newcommand{\bE}{\ensuremath{\mathbb{E}}}

\newcommand{\bN}{\ensuremath{\mathbb{N}}}

\newcommand{\bP}{\ensuremath{\mathbb{P}}}

\newcommand{\bR}{\ensuremath{\mathbb{R}}}

\newcommand{\bZ}{\ensuremath{\mathbb{Z}}}

\newcommand{\ind}{\ensuremath{\mathbbm{1}}}

\newcommand{\cC}{\ensuremath{\mathcal{C}}}

\newcommand{\cE}{\ensuremath{\mathcal{E}}}
\newcommand{\cF}{\ensuremath{\mathcal{F}}}

\newcommand{\bs}{\backslash}

\newcommand{\abs}[1]{\left\vert \, #1 \, \right\vert}
\newcommand{\norm}[1]{\left\Vert \, #1 \, \right\Vert}

\newcommand{\normb}[1]{\interleave \, #1 \, \interleave}

\newcommand{\ddx}[1][1]{\ifnum#1=1 \frac{d}{dx} \else \frac{d^{#1}}{dx^{#1}} \fi}
\newcommand{\ddy}[1][1]{\ifnum#1=1 \frac{d}{dy} \else \frac{d^{#1}}{dy^{#1}} \fi}
\newcommand{\ddt}[1][1]{\ifnum#1=1 \frac{d}{dt} \else \frac{d^{#1}}{dt^{#1}} \fi}

\newcommand{\suml}{\sum\limits}

\newtheorem{theorem}{Theorem}[section]
\newtheorem{lemma}[theorem]{Lemma}
\newtheorem{proposition}[theorem]{Proposition}
\newtheorem{corollary}[theorem]{Corollary}
\newtheorem{definition}[theorem]{Definition}

\newenvironment{remark}[1][Remark]{\begin{trivlist}
\item[\hskip \labelsep {\bfseries #1}]}{\end{trivlist}}

\newcommand{\hP}{\widehat{\bP}}
\newcommand{\hE}{\widehat{\bE}}
\newcommand{\oU}{{\overline{U}}}
\newcommand{\oN}{{\overline{N}}}

\title{A Variance Inequality for Glauber Dynamics applicable to High and Low Temperature Regimes}

\author{Florian V\"ollering\thanks{University of G\"ottingen, IMS, Goldschmidtstra\ss e 7
37077 G\"ottingen,
Germany \newline email: florian.voellering@mathematik.uni-goettingen.de, tel.: +49551 3913520}}

\begin{document}

\maketitle

\begin{abstract}
A variance inequality for spin-flip systems is obtained using comparatively weaker knowledge of relaxation to equilibrium based on coupling estimates for single site disturbances. We obtain variance inequalities interpolating between the Poincar\'e inequality and the uniform variance inequality, and a general weak Poincar\'e inequality. For monotone dynamics the variance inequality can be obtained from decay of the autocorrelation of the spin at the origin, i.e., from that decay we conclude decay for general functions. This method is then applied to the low temperature Ising model, where the time-decay of the autocorrelation of the origin is extended to arbitrary quasi-local functions.
\end{abstract}

\section{Introduction}
Variance estimates and related inequalities have a long history in the study of interacting particle systems. Classical inequalities are the log-Sobolev inequality or Poincar\'e's inequality. A basic distinction between various types of estimates is whether they deal with the mixing structure in space, with respect to some measure, or in time, with respect to some dynamics. It is well-established that strong mixing properties in space imply strong mixing properties in time, and vice versa\cite{MARTINELLI:OLIVIERI:94,HOLLEY:STROOK:87}. Often this connection is made via tensorization arguments of the corresponding inequalities.

In \cite{CHAZOTTES:REDIG:VOELLERING:11} it is shown how a different method, disagreement percolation\cite{BERG:MAES:94}, can be used to obtain a Poincar\'e inequality. The idea used is to track how the influence of a single spin-flip possibly percolates through space, and then use subcriticality of the percolation to obtain results.

All the methods above require in some form uniform estimates. If only weaker mixing properties hold, say in expectation instead of uniform, a lot less is known. One of the few general tools available are Poincar\'e and weak Poincar\'e inequalities (see Section \ref{subsection:Poincare} for more details).

In this paper, we approach the problem of mixing in another direction. We go from a restricted form of decay of correlations in time to general decay of correlations in time. The idea is to track the influence of a single spin-flip through time and space. Even if it is possible that a single flip has a very large influence, it may be that the configurations where that is the case are exceptional, and typically the influence is small. 

Given that an interacting particle system with nearest-neighbour Glauber dynamics satisfies those coupling conditions we obtain variance estimates for the ergodic measures as well as the relaxation of the dynamics. In the case of attractive dynamics, the coupling condition can be relaxed to a condition on the auto-correlation of the spin at the origin. Using the recent progress in \cite{LUBETZKY:MARTINELLI:SLY:TONINELLI:10} on the low-temperature Ising model we can extend the results to obtain quasi-polynomial relaxation to equilibrium of the Glauber dynamics.

\section{Definitions and Notation}
\subsection{Setting}
We consider the state space $\Omega = \{-1,+1\}^{\bZ^d}$. For a function $f:\Omega\to\bR$, which is generally assumed to be bounded and measurable, define
\[ \nabla_x f(\eta) := f(\eta^x)-f(\eta),\quad \eta\in\Omega, x\in\bZ^d, \]
where $\eta^x$ is the configuration $\eta$ flipped at $x$, i.e., $\eta^x(x)=-\eta(x)$ and $\eta^x(y)=\eta(y)$ for $y\neq x$. We call $f$ local if $\nabla_xf = 0$ for all but finitely many $x\in\bZ^d$. In addition, we define a family of semi-norms for functions on $\Omega$, 
\begin{align*}
\normb{f}_p &:= \left( \sum_{x\in\bZ^d} \sup_{\eta\in\Omega} \left|\nabla_x f(\eta)\right|^p \right)^{\frac1p},\quad p\geq1.
\end{align*}

A probability measure $\mu$ on the space $\Omega$ is a called a Markov random field if the probability of observing a plus-spin(or minus-spin) given the spin of all other sites depends only the spin of the nearest neighbours. In terms of a random variable $\xi$ on $\Omega$ that means
\[ \mu\big( \xi(x)=+1 \,\big|\, \forall\,y\neq x : \xi(y)=\eta(y) \big) = \mu\big( \xi(x)=+1 \,\big|\, \forall\,y, \abs{y- x}=1 : \xi(y)=\eta(y) \big) \]
for any $\eta\in\Omega$. With this fact in mind, define
\begin{align*}
c_+(x,\eta) &= \mu\left( \xi(x)=+1 \,\middle|\, \forall\,y\neq x : \xi(y)=\eta(y) \right) ;\\
c_-(x,\eta) &= \mu\left( \xi(x)=-1 \,\middle|\, \forall\,y\neq x : \xi(y)=\eta(y) \right) = 1-c_+(x,\eta) .
\end{align*}
The conditional probabilities are called translation invariant if $c_+(x,\eta) = c_+(0,\tau_x\eta)$ where $\tau_x\eta(y) = \eta(x+y)$.

A natural dynamics with respect to $\mu$ is the Glauber dynamics, where spins at site $x$ flip individually according to some rates $c(x,\eta)$. Here we choose the heat-bath Glauber dynamics, where the flip rates are given by the conditional probabilities $c_+, c_-$:
\begin{align*}
c(x,\eta) := 
\begin{cases}
c_+(x,\eta),\quad&\eta(x)=-1;\\
c_-(x,\eta),\quad&\eta(x)=+1.
\end{cases}
\end{align*}
The associated Markov process $(\eta_t)_{t\geq0}$ is then defined via its generator $L$ acting on the core of local functions,
\begin{align*}
Lf(\eta) &= \suml_{x\in\bZ^d} c(x,\eta)\nabla_x f(\eta).
\end{align*}
The existence of the process is proved for example in \cite[section I.3]{LIGGETT:05}.
Let $\bP_\eta,\eta\in\Omega$, be the path measures on the space of cadlag trajectories and $S_tf(\eta) = \bE_\eta f(\eta_t)$ the corresponding semi-group. We assume the measure $\mu$ to be ergodic with respect to the Glauber dynamics, that is $S_t f = f$ implies $f$ is constant $\mu$-a.s. Note that there can be multiple ergodic measures for the same dynamics.

\subsection{Poincar\'e and uniform variance inequalities}\label{subsection:Poincare}
The Dirichlet form $\cE$ associated to $L$ is given by
\[ \cE(f,f) = -2\int f(\eta)Lf(\eta)\,\mu(d\eta) = \suml_{x\in\bZ^d}\int c(x,\eta)(\nabla_xf)^2(\eta)\,\mu(d\eta). \]
A Poincar\'e inequality is said to hold if for some $K>0$
\begin{align}\label{eq:poincare}
&\Var_\mu(f) \leq K\cE(f,f) = K\sum_{x\in\bZ^d}\int c(x,\eta)(\nabla_x f)^2(\eta)\,\mu(d\eta)
\end{align}
hold for all $f\in L^2(\mu)$. The Poincar\'e inequality is equivalent to a spectral gap of the (self-adjoined) generator $L$ in $L^2(\mu)$ and implies exponential relaxation of the semi-group in $L^2(\mu)$. Under the assumption that $\inf_{\eta\in\Omega}c(\eta,0)>0$, \eqref{eq:poincare} is equivalent to
\begin{align}\label{eq:poincare2}
&\Var_\mu(f) \leq K'\sum_{x\in\bZ^d}\int (\nabla_x f)^2(\eta)\,\mu(d\eta) = K'\sum_{x\in\bZ^d}\norm{(\nabla_x f)^2}_{L^2(\mu)}.
\end{align}

A much weaker inequality is the uniform variance inequality
\begin{align}\label{eq:poincare3}
\Var_{\mu}(f) \leq K'' \normb{f}_2^2 = K''\sum_{x\in\bZ^d}\norm{(\nabla_x f)^2}_\infty.
\end{align}
To the authors knowledge this inequality is not related to any form of relaxation of the semi-group. 

\subsection{Weak Poincar\'e inequality}
When the Poincar\'e inequality does not hold ($K=K'=\infty$) but \eqref{eq:poincare3} is too weak because one still wants to obtain some information about the relaxation speed to equilibrium one can go to other inequalities. One is the so-called weak Poincar\'e inequality, usually formulated as
\begin{align}
\Var_\mu(f) \leq \alpha(r)\cE(f,f) + r \Phi(f), \quad \mu(f)=0, r>0,
\end{align}
where $\Phi(\lambda f)=\lambda^2\Phi(f),\Phi(f)\in[0,\infty]$, and $\alpha$ is a function decreasing to 0. 
This implies the following relaxation to equilibrium:
\[ \Var_\mu(S_T f) \leq \xi(T)\left(\sup_{t\geq 0}\Phi(S_t f) + \Var_\mu(f)\right)  \]
with $\xi(T) = \inf\{ r\geq 0: -\frac{1}{2}\alpha(r)\log(r)\leq T\}$ (see \cite{ROECKNER:WANG:01}).

\section{Main results}
Let $\hP_{\eta,\xi}$ be the basic coupling (based on the graphical construction, see Section \ref{section:graphical-construction}. See also for example \cite[section III.1]{LIGGETT:05}) between two copies of the dynamics starting from the configurations $\eta,\xi\in\Omega$. Set
\begin{align}
\theta_t(\eta) &= \hP_{\eta^0,\eta}(\eta_t^1\neq \eta_t^2) ,\quad t\geq0.
\end{align}
For $p\in[1,\infty]$ define the function $D_p : [0,\infty[ \to [0,\infty]$ as
\[D_p(T) = \int_T^\infty (t+1)^{2d+2} \norm{c_{q}\theta_t}_{L^q(\mu)} \,dt,\]
where $\frac1p + \frac1q =1$ and $c_{q}(\eta)=c(0,\eta)^{\frac1q} \leq 1$. Consequently $\norm{c_q\theta_t}_{L^q(\mu)}\leq\norm{\theta_t}_{L^q(\mu)}$.

The function $D_p$ is going to determine the relaxation speed of $S_tf$ for general functions. Note that by definition $D_p$ is decreasing.

\begin{theorem}\label{thm:main-result}
Let $\mu$ be a translation invariant Markov random field, and $S_t$ the associated heat-bath semi-group. Fix $p\in[1,\infty]$ and assume $D_p(0)<\infty$. For all $f:\Omega\to\bR$ with $\normb{f}_2<\infty$ the following inequality holds:
\begin{align}\label{eq:main-result}
\Var_\mu(S_T f) \leq C_d D_p(T) \sum_{x\in\bZ^d} \norm{(\nabla_x f)^2}_{L^p(\mu)}.
\end{align}
Here $C_d$ is a universal constant depending only on the dimension $d$.
\end{theorem}
\begin{remark}
For $T=0$ we obtain the variance inequality
\begin{align*}
\Var_\mu(f) \leq D_p(0) \sum_{x\in\bZ^d} \norm{(\nabla_x f)^2}_{L^p(\mu)},
\end{align*}
which interpolates between the Poincar\'e inequality ($p=1$) and the uniform variance inequality($p=\infty$).
\end{remark}

In the case of attractive spin-systems it is possible to bound $D_p$ by the auto-correlation of the spin at the origin, which can be easier to estimate than $\theta_t$.
\begin{theorem}\label{thm:auto-correlation}
Assume that the spin-system is attractive. Let $\phi(t):=\Var_\mu(S_t g)$, $g(\eta)=\eta(0)$, be the auto-correlation of the spin at the origin. Then the function $D_p$ can be estimated by 
\[ D_p(T) \leq C_d' \int_T^\infty (t+1)^{3d+2} \left(\phi(t)\right)^{\frac{p-1}{4p}} \,dt,\]
with a dimension dependent constant $C_d'>0$.
\end{theorem}
A good example where this result can be applied is the two-dimensional low-temperature Ising model. Let $\mu^+$ be the plus phase of the 2-dimensional Ising model in the low temperature regime. Recently in \cite{LUBETZKY:MARTINELLI:SLY:TONINELLI:10} the estimate
\[ \Var_{\mu^+}(S_tg) \leq \exp\left(-e^{c(\beta)\sqrt{\log(t+1)}}\right) \]
was obtained, with $c(\beta)$ some temperature dependent constant. Combining this with Theorem \ref{thm:auto-correlation} gives a variance estimate for general functions.
\begin{corollary}
Fix $p>1$. Let $\widetilde{D}_p: [0,\infty[ \to [0,\infty[$ be given by
\[ \widetilde{D}_p(T) = c_p(\beta) \int_T^\infty \exp\left(8\log(t+1)-\frac{p-1}{4p} e^{c(\beta)\sqrt{\log(t+1)}} \right) \,dt.\]
For all $f:\Omega\to\bR$ the relaxation of the semi-group in the plus-phase is estimated by
\[ \Var_{\mu^+}(S_T f) \leq \widetilde{D}_p(T) \sum_{x\in\bZ^d} \norm{(\nabla_x f)^2}_{L^p(\mu)}.  \]
\end{corollary}

\subsection{Extensions to more general settings}
Even though in the present paper the setting is limited to nearest neighbor spin-systems with only $+$ or $-$ spins, this is done more for convenience than out of necessity. It is possible to extend all results to spins in $\{1,...,K\}$ for $K \in \bN$ arbitrary. One then has flip rates $c_k(x,\eta)$ for the rate to go from $\eta$ to $\eta^x_k$, which is the configuration $\eta$ except that at $x$ the value of the spin is replaced by $k$. The coupling condition is then based on $\sum_{k=1}^K  \int c_k(0,\eta) \hP_{\eta^0_k,\eta}(\eta^1\neq\eta^2)^q \mu(d\eta)$.

Everything can also be extended to finite range instead of nearest neighbor interactions. The nearest neighbor interaction is only used to limit the spread of an initial discrepancy to a space-time cone. The same can be done if the interaction is finite range, by changing the partial order in Definition \ref{definition:resampling-order} to allow for interaction paths with $\abs{x_m-x_{m-1}}\leq R$, where $R$ is the interaction range.

\section{Discussion}
\subsection{The coupling parameter $\theta_t$}
The coupling parameter $\theta_t(\eta)=\hP_{\eta^0,\eta}(\eta^1\neq\eta^2)$ is what one needs to control to use Theorem \ref{thm:main-result}. This parameter describes how likely it is that a single flip produces differences which persist up to time $t$. In the literature one can find also other coupling parameters, notably $\rho_t := \sup_{\eta,\xi\in\Omega}\hP_{\eta,\xi}(\eta^1_t(0) \neq \eta^2_t(0))$. If the spin system is attractive, then $\rho_t = S_t g(\overline{+1})-S_t g(\overline{-1})$, where $g(\eta)=\frac12\eta(0)$ and $\overline{+1}$ ($\overline{-1}$) is the configurations with all $+$ ($-$) spins, and hence $\rho_t$ can be seen as a uniform version of $\phi_t$ from Theorem \ref{thm:auto-correlation}.

Comparing $\theta_t$ and $\rho_t$, $\theta_t$ looks at global differences caused by a single spin change while $\rho_t$ looks at differences at a single spin originating from global differences.
The advantage of $\theta_t$ over $\rho_t$ is that we also have the reference spin configuration $\eta$, with respect to which the influence of a spin flip is measured. By distinguishing between uniform estimates with the worst case configuration (the $p=1$ case) and moment estimates (the $p>1$ case) we have more regimes in which the method is applicable. In particular $\theta_t$ is not limited to the uniqueness regime, in contrast to $\rho_t$.

\subsection{The case $p=1$}
For $p=1$ we have the uniform setting, where $\sup_{\eta\in\Omega}\theta_t(\eta)$ decays sufficiently fast. Having uniform control on the coupling  is a strong condition. If $D_1(0)<\infty$ then Theorem \ref{thm:main-result} implies a Poincar\'e inequality and hence exponentially fast relaxation to equilibrium in $L^2$. 
Uniform control on the coupling also allows us to obtain uniform control on the decay of the semi-group.
\begin{proposition}\label{prop:L-infty-decay}
Set $\widehat{D}(T):= \int_T^\infty (t+1)^{d+1}\norm{\theta_t}_{L^\infty(\mu)}^{\frac12}\,dt$ and suppose $\widehat{D}(0)<\infty$. Then for all $f:\Omega\to\bR$ with $\normb{f}_1<\infty$ and all $T\geq 0$
\[ \norm{S_Tf - \mu(f)}_{\infty} \leq C_d^{\frac12} \widehat{D}(T) \normb{f}_1. \]
The constant $C_d$ is the same as in Theorem \ref{thm:main-result}.
\end{proposition}
A consequence of Proposition \ref{prop:L-infty-decay} is that there cannot be two different ergodic measures for the system.
Conditions which guarantee uniqueness of the ergodic measure $\mu$ are well-studied. An important condition of this kind is what is called Dobrushin-Shlosman mixing (DSM). In the literature there are many conditions which imply or  are equivalent to DSM. Just like $D_1(0)<\infty$, DSM also implies a spectral gap. In contrast to Theorem \eqref{thm:main-result}, DSM and equivalent or stronger conditions are all framed in the context of finite volumes. To be more concrete let $\cF$ denote all finite subsets of $\bZ^d$. For $\Lambda \in \cF$ and $\xi\in\Omega$, let $S_t^{V,\xi}$ and $\mu^{V,\xi}$ be the semi-group and stationary measure of the finite volume dynamics in $V$ with boundary condition $\xi$, that is spins outside of $V$ are frozen in configuration $\xi$. Let $\cC$ denote the continuous functions on $\Omega$.
\begin{theorem}[see \cite{YOSHIDA:97}]\label{thm:DSM}
 The following statements are equivalent:
\begin{enumerate}
\item {[DSM]} There are constants $c_1,c_2>0$ so that for any local function $f$,
 \begin{align}\label{eq:DSM}
\sup_{V \in \cF, \xi\in\Omega} \abs{\mu^{\Lambda, \xi^x}(f)-\mu^{\Lambda,\xi}(f)}\leq c_1 \norm{f}_\infty e^{-c_2 d(x,\supp(f))}
\end{align}
\item There exists $C,h:[0,\infty[\to[0,\infty[$ with $h(t)\in o(t^{-2(d-1)})$ as $t\to\infty$ so that for any $f\in\cC$,
\begin{align}\label{eq:DSM-L1}
\sup_{V \in \cF, \xi\in\Omega} \mu^{\Lambda, \xi^x}\left(\abs{S_t^{\Lambda,\xi}f-\mu^{\Lambda,\xi}(f)}\right)\leq C(\norm{f}_\infty+ \abs{\supp(f)}) h(t)
\end{align}
\item There exists $C:[0,\infty[\to[0,\infty[$ and $\lambda>0$ so that for any $f\in\cC$,
\begin{align}\label{eq:DSM-Linfty}
\sup_{V \in \cF, \xi\in\Omega} \norm{S_t^{\Lambda,\xi}f-\mu^{\Lambda,\xi}(f)}_\infty\leq C(\norm{f}_\infty+ \abs{\supp(f)}) e^{-\lambda t}
\end{align}
\end{enumerate}
\end{theorem}
\begin{remark}
It is possible to restrict $\cF$ to a sub-class of finite volumes. Take for example $\cF$ as all sets which consist of unions of big cubes. This is in general strictly weaker than taking all finite volumes. 
\end{remark}
It should should also be remarked that there are many more conditions than the three presented in Theorem \ref{thm:DSM}, see for example \cite{MARTINELLI:OLIVIERI:94}. 


The perhaps remarkable but by now well-known fact that polynomial decay of the semi-group can imply exponential decay is well exhibited by Theorem \ref{thm:DSM}, but was first proved in \cite{HOLLEY:84,AIZENMAN:HOLLEY:87}. In particular it was shown there that in the attractive setting the decay of the coupling parameter $\rho_t$ faster than $t^{-d}$ implies exponential decay of the semi-group. Comparing this to Theorem \ref{thm:auto-correlation}, the difference between the necessary decay rates is obvious. While the term $t^{3d+2}$ is probably not optimal, a direct comparison to the result of \cite{HOLLEY:84} is not possible, since $\rho_t$ is a term uniform over all configurations, while Theorem \ref{thm:auto-correlation} requires $p>0$. Also, the dichotomy between either exponential or slow decay no longer exists in the non-uniform case, as the low-temperature Ising model shows. 

It is natural to compare the conditions of \ref{thm:main-result} to DSM, particularly condition \eqref{eq:DSM-L1}. Instead of the coupling condition $D_p(0)<\infty$ an $L^1$-mixing condition in finite volumes is required, also with a polynomial decay. However, the supremum over finite volumes with a uniform control on the boundary is significant. Even though \eqref{eq:DSM-L1} looks more like an $L^1$-condition, since it is equivalent to \eqref{eq:DSM-Linfty} it belongs to the uniform case. In the case of attractive spin-systems it does imply that $\theta_t(\eta)$ decays exponentially fast uniformly in $\eta$:
\begin{proposition}\label{prop:DSM}
If the spin-system is attractive and DSM is satisfied, then there are constants $C,\lambda>0$ so that $\norm{\theta_t}_{L^\infty(\mu)} \leq C e^{-\lambda t}$. Particularly, $D_1(0)<\infty$.
\end{proposition}
It is not clear whether $D_1(0)<\infty$ implies DSM. If $\norm{\theta_t}_{L^\infty(\mu)}$ decays exponentially fast, then so does by Proposition \ref{prop:L-infty-decay} the semi-group in the supremum-norm. Except for the difference between finite and infinite volume it is like condition \eqref{eq:DSM-Linfty}. However, there is in general a strict difference between finite and infinite volume conditions as demonstrated by the so-called Czech models, which exhibit a phase transition in the half-space but not in the full space \cite{SHLOSMAN:86}. 

To conclude the discussion of the uniform case $p=1$, we provide a simple condition on the flip rates $c(x,\eta)$ so that $D_1(0)<\infty$. 
\begin{proposition}\label{prop:uniform-condition}
Suppose 
\begin{align}\label{eq:uniform-condition}
\alpha:=\sup_{\eta\in\Omega}\sum_{\abs{x}=1}\abs{c(x,\eta^0)-c(x,\eta)}<1 .
\end{align}
Then 
$ \norm{\theta_t}_{L^\infty(\mu)}\leq e^{-(1-\alpha)t} .$
\end{proposition}
Condition \eqref{eq:uniform-condition} compares favorably to the well-known Dobrushin uniqueness criterion \cite{DOBRUSCHIN:68} or Ligget's $M<\epsilon$ regime \cite[page 123]{LIGGETT:05}, which in this context both read
\[ \sum_{\abs{x}=1}\sup_{\eta\in\Omega}\abs{c(0,\eta^x)-c(0,\eta)}<1. \]

\subsection{The case $p>1$}
The example of the low-temperature Ising model shows that $D_p(0)<\infty$ does not imply uniqueness of the ergodic measure $\mu$ nor exponentially fast convergence to equilibrium. As such $D_1(0)<\infty$ and $D_p(0)<\infty,p>1$, are very distinct and exponential decay of the semi-group in $L^2$ is not guaranteed in the $p>1$ case. However, if $\norm{\theta_t}_{L^q(\mu)}$ decays exponentially fast Theorem \ref{thm:main-result} still implies exponentially fast decay of the variance, but with respect to a stronger norm than $\norm{\cdot}_{L^2(\mu)}$. This, however, is sufficient to prove a spectral gap of the generator $L$ or, equivalently, a Poincar\'e inequality.
\begin{proposition}\label{prop:spectral-gap}
Suppose $\int_0^\infty \norm{\theta_t}_{L^q(\mu)} e^{\lambda t}\,dt<\infty$ for some $\lambda>0$ and $1\leq q \leq \infty$. Then $]-\lambda/2,0[$ belongs to the resolvent set of $L$.
\end{proposition}
In fact, we can say even more about the connection between $\norm{\theta_t}_{L^q(\mu)}$ and the Poincar\'e inequality.
\begin{proposition}\label{prop:reverse-spectral-gap}
Suppose the spin system is attractive and $\inf_{\eta\in\Omega}c(\eta,0)>0$. If the spin system satisfies the Poincar\'e inequality, then $\norm{\theta_t}_{L^q(\mu)}$ decays exponentially fast for any $1\leq q < \infty$.
\end{proposition}
This shows that for attractive spin systems equivalence between exponential decay of $\norm{\theta_t}_{L^q(\mu)}, 1\leq q<\infty,$ and the existence of a spectral gap. However, is makes no statement about $\norm{\theta_t}_{L^\infty(\mu)}$.

In the general (that is, not necessarily attractive) case we are left to discuss the situation where $D_1(0)=\infty$ and $D_p$ decays sub-exponentially fast.  In that case it is natural to compare Theorem \ref{thm:main-result} with a weak Poincar\'e inequality.  
\begin{proposition}\label{prop:weak-poincare}
Assume the conditions of Theorem \ref{thm:main-result} and $c(\eta,x)\geq \delta>0$. Then for all $f:\Omega\to\bR$ with $\sum_{x\in\bZ^d}\norm{(\nabla_xf)^2}_{L^p(\mu)}$ and all $t\geq 0$,
\begin{align*}
&\Var_\mu(S_t f)\leq C_d \delta^{-1} D_1(0,R)\cE(f,f) + D_p(R)\Phi_{R}(f),
\end{align*}
where 
\begin{align*}
D_p(0,R) &= \int_0^R (t+1)^{2d+2} \norm{\theta_t}_{L^p(\mu)},\\
\Phi_R( f) &= \frac{\Var_{\mu}(S_{R}f)}{D_p(R)} \leq C_d \suml_{x\in\bZ^d}\norm{(\nabla_xf)^2}_{L^p(\mu)}.
\end{align*}
\end{proposition}
This weak Poincar\'e inequality leads (with a minor modification of the proof in \cite{ROECKNER:WANG:01}) to 
\[ \Var_\mu(S_Tf) \leq \xi(T) \left(C_d \suml_{x\in\bZ^d}\norm{(\nabla_xf)^2}_{L^p(\mu)} + \Var_\mu(f)\right). \]
The decay $\xi(T)$ is of order $D_p(T^{\frac{1}{2d+3}})$, which is worse than the one from Theorem \ref{thm:main-result}. The reason for that is that in the weak Poincar\'e inequality the diverging $D_1(0,R)$ is partially used, while Theorem \ref{thm:main-result} makes only use of the converging $D_p(0,R)$.

\section{Graphical construction}\label{section:graphical-construction}
The graphical construction of the Glauber heat bath dynamics is the encoding of the random evolution of the process $\eta_t$ into basic random components and a deterministic function of this randomness and the initial configuration. It is a well-known tool in the study of spin and particle systems.

Let $\oN$ be a Poisson point process on $\bZ^d\times[0,\infty[$ with intensity one(wrt. the counting measure on $\bZ^d$ and the Lebesgue measure on $[0,\infty[$). A point $(x,t)\in\oN$ represents a \emph{chance} of flipping the spin at site $x$ and time $t$. To realize this chance let $\oU=(\oU_n)_{n\in\bN}$ be a countable iid. collection of $[0,1]$-uniform random variables independent of $\oN$. We assume that to each $(x,t)\in\oN$ there is an associated $U$ from $\oU$ (which can be realized by a bijection from $\oN$ to $\bN$, and we simply write $\oU:\oN\to[0,1]$). We denote the expectation with respect to $\oN$ and $\oU$ by $\int d\oN$ and $\int d\oU$.

The elementary step is then as follows. Given the configuration $\eta_{t-}$ before a possible flip at $(x,t)\in\oN$ and the to $(x,t)$ associated random variable $U=\oU((x,t))$  we determine the configuration $\eta_t$ after the possible flip deterministically. All sites $y\in\bZ^d,y\neq x,$ are unchanged, i.e., $\eta_t(y)=\eta_{t-}(y)$. If $U<c_+(x,\eta_{t-})$, then $\eta_t(x)=+1$, otherwise $\eta_t(x)=-1$. Since we ignore the original spin at $x$ and simply replace it with a new one drawn according to conditional probability given the other spins we call this a \emph{resampling event}.

The configuration $\eta_t$ is then given by the successive application of all resampling events to the initial configuration $\eta_0$. As those are infinitely many steps one has to take care that this is indeed well-defined. The goal is to define a deterministic function $\Psi$ which will output the configuration at time $t$, $\eta_t$, given the inputs $\oN,\oU$ and $\eta_0$. We now focus on the precise construction of the graphical representation and its properties.

For a single resampling event the definition of $\Psi$ is simple. Let $\Psi:\Omega\times (\bZ^d\times[0,\infty[\times[0,1])\to \Omega$ be given by
\begin{align*} 
\Psi(\eta,(x,t,u)) (y) := 
\begin{cases}
+1,\quad &y=x, c_+(x,\eta)\leq u;	\\
-1,\quad &y=x, c_+(x,\eta)>u;	\\
\eta(y),\quad &y\neq x.
\end{cases} 
\end{align*}
This definition is directly extended recursively to a finite number of resampling events. For $(x_n,t_n,u_n)_{1\leq n\leq N} \subset \bZ^d\times [0,\infty[\times[0,1]$ with $t_1<t_2<...<t_N$,
\begin{align*}
&\Psi\left(\eta, (x_n,t_n,u_n)_{1\leq n\leq N}\right) := \Psi\left(\Psi\left(\eta,(x_1,t_1,u_1)\right), (x_n,t_n,u_n)_{2\leq n\leq N}\right),
\end{align*}
and $\Psi(\eta,\emptyset)=\eta$. 
\begin{definition}\label{definition:resampling-order}
Let $G$ be a countable subset of $\bZ^d\times[0,\infty[$.
\begin{enumerate}
\item A partial order $<_G$ on $\bZ^d\times[0,\infty[$ is defined as follows: $(x,t)<_G (y,s)$ iff either $x=y$ and $t<s$ or there exists a finite subset $\{(x_1,t_1),\ldots(x_K,t_K)\}\subset G$ such that $t<t_{1}<t_{2}<\ldots<t_{K}\leq s$ and $\abs{x_{m}-x_{{m-1}}}=1$, $2\leq m \leq K$, as well as $\abs{x_{1}-x}=1$ and $x_{K}=y$.
\item Write $T_x:=\sup\{t:(x,t)\in G\},x\in\bZ^d,$ and 
$G_{<x}:=\{(y,t)\in G:(y,t)\leq_G (x,T_x)\}$. We call $G$ \emph{locally finite}, if 
$\abs{G_{<x}}<\infty$ for all $x\in\bZ^d$.
\item For $G^U$ a countable subset of $\bZ^d\times[0,\infty[\times[0,1]$ the definitions a) and b) are copied in the canonical way(projection of $G^U$ onto $\bZ^d\times[0,\infty[$).
\end{enumerate}
\end{definition}
The purpose of this definition becomes transparent by the following fact.
\begin{lemma}
For any $G^U \subset \bZ^d\times[0,\infty[\times[0,1]$ finite, $x\in\bZ^d$ and $\eta\in\Omega$,
\[ \Psi(\eta, G^U)(x) = \Psi(\eta,G^U_{<x})(x) .\]
\end{lemma}
\begin{proof}
The nearest-neighbour property of $c_+$ means that to determine the new spin after a resampling event $(x,t)$ it is sufficient to know the spin value of the neighbours of $x$. Those might depend on earlier resampling events, which have again nearest neighbour dependencies, and all resampling events $(y,s)$ which have an influence on $(x,t)$ satisfy $(y,s)<_G(x,t)$. 
\end{proof}
This leads is to the final definition of $\Psi$. For $G^U$ a locally finite subset of $\bZ^d\times[0,\infty[\times[0,1]$ (or $G\subset \bZ^d\times[0,\infty[, U:G\to[0,1], G^U:=\{(x,t,U(x,t)):(x,t)\in G\}$),
\[ \Psi(\eta,G^U)(x) := \Psi(\eta,G^U_{<x})(x),\quad x\in\bZ^d. \]
An important property of the graphical construction evident here is that $\Psi$ is tolerant to certain changes in the order of resampling events. Intuitively, a resampling event $(x,t)$ is influenced only by resampling events which happen before $t$ and are not too distant from $x$. This intuition can be formalized via the ordering $>_G$, which we now do.
\begin{lemma}\label{lemma:reordering}
Let $G^U\subset \bZ^d\times[0,\infty[\times[0,1]$ be locally finite and $A,B \subset G^U$ a partition of $G^U$ such that $\forall\, (x_1,t_1,u_1)\in A, (x_2,t_2,u_2)\in B: (x_1,t_1) \ngtr_{G} (x_2,t_2)$. In words, $A$ does not happen after $B$. Then
\[ \Psi(\eta, G^U) = \Psi\left(\Psi(\eta,A), B\right). \]
\end{lemma}
\begin{proof}
Assume $G^U$ is finite. If not, restrict to $G^U_{<x}$.

The proof is a consequence from the following basic fact. For $(x_i,t_i,u_i)\in \bZ^d\times[0,\infty[\times[0,1],i=1,2$, with $\abs{x_1-x_2}>1$,
\begin{align}
\Psi(\eta,\{(x_1,t_1,u_1),(x_2,t_2,u_2)\}) &= \Psi(\Psi(\eta,(x_1,t_1,u_1)),(x_2,t_2,u_2)). \label{eq:reorderingfact}
\end{align}

By the property of the decomposition for each $(x_1,t_1,u_1)\in A, (x_2,t_2,u_2)\in B$, either $t_1<t_2$ or $\abs{x_1-x_2}>1$. The proof of the lemma is an iterative application of fact \eqref{eq:reorderingfact}.
Let $a_i, i=1..\abs{A}$ be the elements of $A$ ordered in increasing time. Starting from $\Psi(\eta, A\cup B) = \Psi(\Psi(\eta,\emptyset),\{a_i : i=1,...,\abs{A} \}\cup B)$, we can use fact \eqref{eq:reorderingfact} to move $a_1$ past all resampling events in $B$ and perform this resampling event first:
\[ \Psi(\eta, A \cup B) = \Psi(\Psi(\eta,\{a_1\}), \{a_i : i=2,...,\abs{A} \}\cup B). \]
\begin{figure}[hbtp]
\begin{center}
\setlength{\unitlength}{1cm}
\begin{picture}(7,5.5)(0,0)
\put(0,0){\line(1,0){7}}
\put(0,5){\line(1,0){7}}

\put(1,0){\line(0,1){5}}
\put(2,0){\line(0,1){5}}
\put(3,0){\line(0,1){5}}
\put(4,0){\line(0,1){5}}
\put(5,0){\line(0,1){5}}
\put(6,0){\line(0,1){5}}

\put(1,4){\makebox(0,0){$\times$}}
\put(1,4){\makebox(0.6,0){$a_5$}}
\put(1,1){\makebox(0,0){$\times$}}
\put(1,1){\makebox(0.6,0){$a_2$}}
\put(2,2){\makebox(0,0){$\times$}}
\put(2,2){\makebox(0.6,0){$a_3$}}
\put(3,0.2){\makebox(0,0){$\times$}}
\put(3,0.2){\makebox(0.6,0){$a_1$}}
\put(6,2.1){\makebox(0,0){$\times$}}
\put(6,2.1){\makebox(0.6,0){$a_4$}}

\put(3,2.6){\circle*{.1}}
\put(3,2.6){\makebox(0.6,0){$b_2$}}
\put(3,4.5){\circle*{.1}}
\put(3,4.5){\makebox(0.6,0){$b_4$}}
\put(4,1.2){\circle*{.1}}
\put(4,1.2){\makebox(0.6,0){$b_1$}}
\put(5,3.7){\circle*{.1}}
\put(5,3.7){\makebox(0.6,0){$b_3$}}
\put(6,4.8){\circle*{.1}}
\put(6,4.8){\makebox(0.6,0){$b_5$}}

\end{picture}
\end{center}
\caption{\small Resampling events $a_1,...,a_5$ do not depend on $b_1,...,b_5$.}
\end{figure}
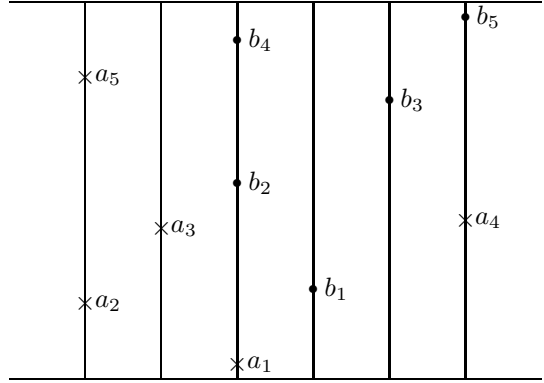

Repeating this procedure for all other elements of $A$ in their time-order then proves the claim of the lemma.
\end{proof}

The final proposition of this section sums up the properties of the graphical representation.
\begin{proposition}\label{prop:graphical-construction}
Let $f:\Omega\to\bR$ be quasi-local (that is $f$ can be uniformly approximated by local functions). 
The function $\Psi$ has the following properties:
\begin{enumerate}
\item $\int \int f\left(\Psi(\eta,\oN_t^\oU)\right) d\oU d\oN = S_t f(\eta)$ , 
where $\oN_t^\oU = \{ (x,s,u) \in \oN^\oU : s\leq t \} $; 
\item For any locally finite $G \subset \bZ^d\times[0,\infty[$, $\int \int f(\Psi(\eta,G^\oU)) \,d\oU \mu(d\eta) = \int f(\eta) \,\mu(d\eta)$;
\item For $\eta^1,\eta^2\in\Omega$ the coupling $\hP_{\eta^1,\eta^2}$ of $\bP_{\eta^1}$ and $\bP_{\eta^2}$ is defined via
\[ \hE_{\eta^1,\eta^2}f(\eta_t^1,\eta_t^2) = \int\int f\left(\Psi(\eta^1,\oN_t^\oU),\Psi(\eta^2,\oN_t^\oU)\right)\,d\oU d\oN. \]
\end{enumerate}
\end{proposition}
\begin{proof}
\par{a)} The point process $(\oN_t^\oU)_{t\geq 0}$ is a Markov process on the subsets of $\bZ^d\times[0,\infty[\times[0,1]$ under $d\oU\,d\oN$ and with respect to the canonical filtration. The image process
\[ \tilde\eta_t:= \Psi(\eta,\oN_t^\oU) \]
is also a Markov process since $\Psi$ preserves the Markov property:
\[ \tilde\eta_t = \Psi\left(\eta,\oN_t^{\oU}\right) = \Psi\left(\Psi(\eta, \oN_s^\oU), \oN_t^\oU \bs \oN_s^\oU\right) = \Psi\left(\tilde\eta_s, \oN_t^\oU \bs \oN_s^\oU\right), \quad t>s\geq0. \]
The generator of $\tilde\eta_t$ is
\begin{align}\label{eq:generator-graphical-representation} 
\widetilde{L} f(\eta) = \sum_{x\in\bZ^d} \int_0^1 f\left(\Psi(\eta,(x,0,u))\right)-f(\eta)\,du.	
\end{align}
Since $\Psi(\eta,(x,0,u))$ is either $\eta$ or $\eta^x$, after integrating over $u$ we obtain $\widetilde{L}f=Lf$ on the core of local functions $f:\Omega\to\bR$.
\par{b)} The proof follows the construction of $\Psi$. Let $G=\{(x,t)\}$ and write $\eta^x_+(x)=+1$,  $\eta^x_+(y) = \eta(y)$ for $y\neq x$ ($\eta^x_-$ analogue). Then 
\begin{align*}
\int \int f\left(\Psi(\eta,G^\oU)\right)\,d\oU\,\mu(d\eta) 
&= \int \int_0^1 f\left(\Psi(\eta,(x,t,u))\right)\,du\,\mu(d\eta) \\
&= \int c_+(x,\eta)f(\eta^x_+) + c_-(x,\eta)f(\eta^x_-)\,\mu(d\eta) \\
&= \int f(\eta)\,\mu(d\eta).
\end{align*}
For $G$ a finite set the result is true by the iterative construction. For $G$ countable but locally finite we observe that for local $f$ only finitely many resampling steps have to be performed to determine the expectation of $f$.
\par{c)} By part a) $\hE_{\eta^1,\eta^2}f(\eta_t^1) = S_tf(\eta^1)$ and $\hE_{\eta^1,\eta^2}f(\eta_t^2)=S_tf(\eta^2)$, so $\hP_{\eta^1,\eta^2}$ is indeed a coupling.
\end{proof}

\section{Proofs of the results}
The first step is to rewrite the variance. As the following formula holds fairly generally and not just in this setting we formulate the lemma with more abstract conditions.
\begin{lemma}\label{lemma:variance}
Let $\mu$ be an ergodic measure wrt. $S_t$ and $f:\Omega \to \bR$ such that $S_t f, (S_t f)^2 \in dom (L)$. Then, for $0\leq T < S \leq \infty$,
\begin{align} \label{eq:variance}
\Var_\mu(S_T f) - \Var_\mu(S_S f) &= \int_T^S \int \left[L(S_tf-S_t f(\eta))^2\right](\eta)\,\mu(d\eta)\,dt \\
&= \int_T^S \int \suml_{x\in\bZ^d} c(x,\eta) \left(S_tf(\eta^x)-S_t f(\eta)\right)^2\,\mu(d\eta)\,dt. \label{eq:variance2}
\end{align}
Note that by ergodicity $\lim_{S\to\infty}\Var_\mu(S_S f)=0$.
\end{lemma}
\begin{proof}
Since
\[\frac{d}{dt} \Var_\mu(S_t f) = \int 2 S_t f(\eta) L S_t f(\eta)\,\mu(d\eta), \]
we can express the variance as
\[ \Var_\mu(S_T f)-\Var_\mu(S_S f) = \int_T^S \int -2 S_t f(\eta) LS_t f(\eta)\,\mu(d\eta)\,dt. \]
By stationarity, $\int [L(S_t f)^2](\eta)\, \mu(d\eta) = 0$, hence
\begin{align*}
\Var_\mu(S_T f)-\Var_\mu(S_S f) &= \int_T^S \int [L(S_t f)^2](\eta) -2 S_t f(\eta) LS_t f(\eta)\,\mu(d\eta)\,dt \\
&= \int_T^S \int [L(S_t f - S_tf(\eta))^2](\eta)\,\mu(d\eta)\,dt.
\end{align*}
\end{proof}
Note that in the setting of Glauber dynamics $\normb{f}_1<\infty$ implies both $\normb{S_t f}, \normb{(S_t f)^2}<\infty$, which in turn implies $S_tf, (S_tf)^2 \in\dom(L)$ \cite[section I.3]{LIGGETT:05}.

The idea of the proof of Theorem \ref{thm:main-result} is to rewrite \eqref{eq:variance} using the graphical representation to describe the semi-group $S_t$. Then various applications of H\"older's inequality are used to separate different parts contributing to the variance formulation \eqref{eq:variance}. However the calculation is fairly sensitive to the order in which different aspects are treated, and has one crucial non-trivial use of the graphical construction on the infinite volume.

We start by looking how the graphical construction can be used in light of Lemma \ref{lemma:variance}.
Let, by slight abuse of notation, $\oN \subset \bZ^d\times\,[0,\infty[$ be a fixed realization of the Poisson point process on $\bZ^d\times\,[0,\infty[$, the set of resampling events. Almost surely this is a locally finite subset of $\bZ^d\times\,]0,\infty[$. We denote all resampling events up to time $t$ by $\oN_{t} := \{(y,s)\in \oN\,:\,s\leq t\}$.

To determine what influence a flip at site $x$ has on the configuration at time $t$ we use the graphical construction, particularly the partial order introduced in definition \ref{definition:resampling-order}. Given the fixed realization $\oN$, the cone
\begin{align*}
C_{t,x} &:= \{(y,s) \in \oN_t : (y,s)>_\oN (x,0) \}
\end{align*}
contains all resampling events which depend on the value of the initial configuration at site $x$, see also figure \ref{fig:cone}. 
\begin{figure}[hbtp]
\begin{center}
\setlength{\unitlength}{1cm}
\begin{picture}(7,5.5)(0,0)
\put(0,0){\line(1,0){7}}
\put(0,5){\line(1,0){7}}

\put(1,0){\line(0,1){5}}
\put(2,0){\line(0,1){5}}
\put(3,0){\line(0,1){5}}
\put(4,0){\line(0,1){5}}
\put(5,0){\line(0,1){5}}
\put(6,0){\line(0,1){5}}

\put(1,4){\makebox(0,0){$\times$}}
\put(1,1){\makebox(0,0){$\times$}}
\put(2,0.3){\makebox(0,0){$\times$}}
\put(2,2){\makebox(0,0){$\times$}}
\put(6,2.1){\makebox(0,0){$\times$}}

\put(2,4.2){\circle*{.1}}
\qbezier[15](2,4.2)(2.5,4.2)(3,4.2)
\qbezier[15](1.95,4.2)(1.95,4.6)(1.95,5)
\put(3,2.6){\circle*{.1}}
\qbezier[15](3,2.6)(3.5,2.6)(4,2.6)
\qbezier[15](2.95,2.6)(2.95,3.4)(2.95,4.2)
\qbezier[15](3.95,0)(3.95,1.3)(3.95,2.6)
\qbezier[15](4.05,0)(4.05,1)(4.05,2)
\qbezier[15](4.05,2)(4.05,2.85)(4.05,3.7)
\qbezier[15](5.05,3.7)(5.05,4.25)(5.05,4.8)
\qbezier[15](6.05,4.8)(6.05,4.9)(6.05,5)
\qbezier[15](4,3.7)(4.5,3.7)(5,3.7)
\qbezier[15](5,4.8)(5.5,4.8)(6,4.8)

\put(3,4.5){\circle*{.1}}
\put(4,1.2){\circle*{.1}}
\put(5,3.7){\circle*{.1}}
\put(6,4.8){\circle*{.1}}

\put(-0.3,0){\makebox(0,0){$0$}}
\put(-0.3,5){\makebox(0,0){$t$}}
\put(4,-0.3){\makebox(0,0){$x$}}

\end{picture}
\end{center}
\caption{\small The cone $C_{t,x}$ containing all resampling events depending on $(0,x)$.}\label{fig:cone}
\end{figure}
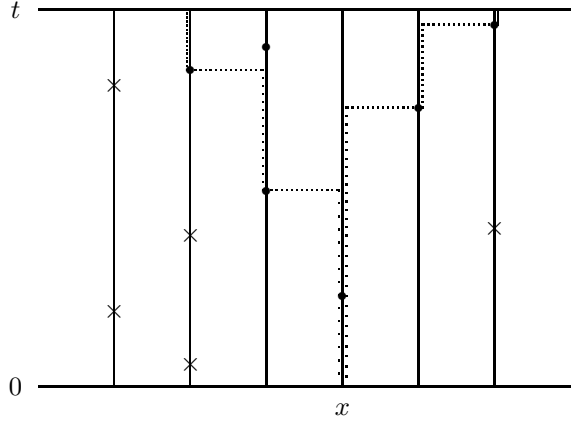
Motivated by \eqref{eq:variance2} we also introduce the same cone with another resampling event added at site $x$ and time 0:
\[ \widetilde{C}_{t,x} := C_{t,x}\cup \{x,0\}. \]
Given a realization of the independent uniform $[0,1]$ variables associated to the resampling events, $\oU:\oN\to[0,1]$, we extend the above sets to
\begin{align*}
\oN_{t}^\oU &:= \{(y,s,\oU((y,s))): (y,s)\in \oN_t\}; \\ 
C_{t,x}^\oU &:= \{(y,s,\oU((y,s))): (y,s)\in C_{t,x}\}. 
\end{align*}
In the case of the added resampling event at $(x,0)$ we assume a given $u\in[0,1]$ to extend the event to $(x,0,u)$. This leads to
\begin{align*}
\widetilde{\oN}_t^\oU &:= \oN_t^\oU \cup \{(x,0,u)\}, \\
\widetilde{C}_{t,x}^\oU &:= C_{t,x}^\oU \cup \{(x,0,u)\},
\end{align*}
and, from $\eta\in\Omega$, 
\begin{align*}
\widetilde{\eta} &:= \Psi(\eta,(x,0,u)). 
\end{align*}
Now we are ready to formulate the crucial idea. We want to compare the evolution of two configurations $\eta^1_t,\eta^2_t$ under the graphical construction coupling when
started from two initial configurations $\eta,\widetilde\eta$. By the graphical construction,
\begin{align*}
\eta^1_t &= \Psi(\eta, \oN^\oU_t), \\
\eta^2_t &= \Psi(\widetilde\eta, \oN^\oU_t) = \Psi(\eta, \widetilde{\oN}^\oU_t). 
\end{align*}
By the reordering principle of the graphical construction in Lemma \ref{lemma:reordering},
\begin{align}
\eta^1_t &= \Psi(\xi, C^\oU_{t,x}), \label{eq:eta1-xi}\\
\xi&= \Psi(\eta, \oN^\oU_t \bs C^\oU_{t,x}). \nonumber
\end{align}
Similarly, 
\begin{align}\label{eq:eta2-xi}
 \eta^2_t = \Psi(\xi, \widetilde{C}_{t,x}^\oU). 	
\end{align}

So we can see $\xi$ as a common ancestor of $\eta^1_t$ and $\eta^2_t$ in terms of the graphical construction (it is not an ancestor in time). This is very important, as both configurations only differ from $\xi$ by a finite number of resampling events, namely those in $C_{t,x}^\oU$ or $\widetilde{C}_{t,x}^\oU$ respectively. The proof of Theorem \ref{thm:main-result} is based on this observation, with Lemma \ref{lemma:variance} as a starting point.

To further facilitate the comparison of $\eta^1_t,\eta^2_t$ with $\xi$, write $C_{t,x}$ as the enumeration $\{ (x_k,t_k,U_k), 1\leq k \leq \abs{C_{t,x}} \}$ with $t_k\geq t_{k-1}$ and $(x_0,t_0,U_0) = (x,0,u)$. With this, 
\begin{align}
\xi_k &:= \Psi(\xi_{k-1},(x_k,t_k,U_k)), \quad 1\leq k\leq \abs{C_{t,x}},	\label{eq:xi_k}\\
\xi_0 &:= \xi, \nonumber\\
\widetilde{\xi}_k &:= \Psi(\widetilde{\xi}_{k-1},(x_k,t_k,U_k)), \quad 1\leq k\leq \abs{C_{t,x}},	\label{eq:tildexi_k}\\
\widetilde{\xi}_0 &:= \Psi(\xi, (x,0,u)).	\nonumber
\end{align}
By Proposition \ref{prop:graphical-construction} $\xi_k$, $\widetilde{\xi}_k$ are $\mu$-distributed provided $\xi$ is since they are obtained via resampling steps. So we can describe $\eta^1_t$ and $\eta^2_t$ via finitely many flips from a common ancestor $\xi$, and each step in between is $\mu$-distributed.

With the observations above we can rewrite part of \eqref{eq:variance2} using the graphical representation.
\begin{lemma}\label{lemma:blowup}
Using above notation,
\begin{align*}
& \left(S_tf(\widetilde{\eta}) - S_t f(\eta) \right)^2 \\
&\quad \leq \hP_{\widetilde{\eta},\eta}\left(\eta^1_t\neq \eta^2_t\right)\int d\oN \left(2\abs{C_{t,x}}+1\right) \int d\oU \\
&\qquad \left[ \suml_{k=1}^{\abs{C_{t,x}}} \left(\nabla_{x_k}f(\widetilde{\xi}_{k-1})\right)^2 + \suml_{k=1}^{\abs{C_{t,x}}} \left(\nabla_{x_k}f({\xi}_{k-1})\right)^2 + \left(\nabla_{x}f(\xi)\right)^2 \right] .
\end{align*}
\end{lemma}
\begin{proof}
Start with
\begin{align*}
\left(S_t(\tilde\eta)-S_tf(\eta)\right)^2 
&= \left(\hE_{\tilde\eta,\eta}(f(\eta^1_t)-f(\eta^2_t))\ind_{\eta^1_t\neq\eta^2_t}\right)^2 \\
&\leq \hE_{\tilde\eta,\eta}\left(f(\eta^1_t)-f(\eta^2_t)\right)^2\,\hP_{\tilde\eta,\eta}(\eta^1_t \neq \eta^2_t ).
\end{align*}
Now let $\hP$ be the graphical construction coupling, then, in the notation of Section \ref{section:graphical-construction},
\[ \hE_{\widetilde\eta,\eta}(f(\eta^1_t)-f(\eta^2_t))^2 = \int d\oN \int d\oU \left[f\left(\Psi(\tilde\eta,\oN_t^\oU)\right) - f\left(\Psi(\eta,\oN_t^\oU)\right)\right]^2. \]
Using \eqref{eq:eta1-xi} and \eqref{eq:eta2-xi},
\begin{align}
\left[f\left(\Psi(\tilde\eta,\oN_t^\oU)\right) - f\left(\Psi(\eta,\oN_t^\oU)\right)\right]^2
= \left[f\left(\Psi(\xi,\widetilde{C}_{t,x}^\oU)\right) - f\left(\Psi(\xi, C_{t,x}^\oU)\right)\right]^2. \label{eq:difference-from-xi}
\end{align}
This can be rewritten using the telescopic sum over the individual resampling steps \eqref{eq:xi_k},\eqref{eq:tildexi_k}: 
\begin{align*}
f\left(\Psi\left(\xi,C_{t,x}^\oU\right)\right) - f(\xi_0) &= \suml_{k=1}^{\abs{C_{t,x}}} f(\xi_k)-f(\xi_{k-1}),	\\
f\left(\Psi\left(\xi,\widetilde{C}_{t,x}^\oU\right)\right) - f(\tilde\xi_0) &= \suml_{k=1}^{\abs{C_{t,x}}} f(\tilde{\xi}_k)-f(\tilde{\xi}_{k-1}) .
\end{align*}
Putting the telescopic sums into \eqref{eq:difference-from-xi} and using the inequality $(\sum_{i=1}^n a_i)^2\leq n\sum_{i=1}^n a_i^2$ leads to the upper bound
\begin{align*}
\left( 2\abs{C_{t,x}} + 1 \right)& \left[ \suml_{k=1}^{\abs{C_{t,x}}} \left(f(\tilde{\xi}_k)-f(\tilde{\xi}_{k-1})\right)^2 \right. \\
&+ \left. \suml_{k=1}^{\abs{C_{t,x}}} \left(f(\xi_k)-f(\xi_{k-1})\right)^2 + (f(\tilde\xi_0)-f(\xi_0))^2 \right].
\end{align*}
Notice that by construction, $\xi_k$ and $\xi_{k-1}$ are identical except for a possible flip at site $x_k$. Consequently, we can further estimate by
\[ \left( 2\abs{C_{t,x}} + 1 \right)\left[ \suml_{k=1}^{\abs{C_{t,x}}} (\nabla_{x_k}f({\tilde\xi}_{k-1}))^2 + \suml_{k=1}^{\abs{C_{t,x}}} (\nabla_{x_k}f(\xi_{k-1}))^2 + (\nabla_x f(\xi))^2 \right]. \]
\end{proof}
The next lemma deals with rearranging and separating integrals as well as condensing the individual terms as much as possible, continuing where Lemma \ref{lemma:blowup} left off.
\begin{lemma}\label{lemma:hardwork}
For $1\leq p \leq q \leq \infty$ with $\frac1p +\frac1q =1$,
\begin{align*}
&\sum_{x\in\bZ^d} \int \mu(d\eta) \int_0^1 du \left( S_t f(\Psi(\eta, (x,0,u))) - S_t f(\eta) \right)^2 \\
&\quad \leq \left(\int c(0,\eta) \theta_t(\eta)^q\,\mu(d\eta)\right)^{\frac1q}\left( \int \left(2\abs{C_{t,0}}+1\right)^2 d\oN \right)\left( \sum_{x\in\bZ^d} \norm{(\nabla_x f)^2}_{L^p(\mu)}\right),
\end{align*}
where 
\begin{align*}
\theta_t(\eta) &= \hP_{\eta^0,\eta}(\eta_t^1\neq \eta_t^2).
\end{align*}
\end{lemma}
\begin{proof}
We start by using Lemma \ref{lemma:blowup} to estimate the inner term of
\[ \sum_{x\in\bZ^d} \int \mu(d\eta) \int_0^1 du \left( S_t f(\Psi(\eta, (x,0,u))) - S_t f(\eta) \right)^2. \]
Upon reordering some of the integrals and sums, we obtain
\begin{align}
&\sum_{x\in\bZ^d} \int d\oN \left(2\abs{C_{t,x}}+1\right)  \nonumber\\
&\qquad \left[ \suml_{k=1}^{\abs{C_{t,x}}} \int \mu(d\eta) \int_0^1 du \int d\oU  \left(\nabla_{x_k}f(\widetilde{\xi}_{k-1})\right)^2 \hP_{\widetilde{\eta},\eta}\left(\eta^1_t\neq \eta^2_t\right) \right.	\label{eq:hardwork1a}\\
&\qquad+ \suml_{k=1}^{\abs{C_{t,x}}} \int \mu(d\eta) \int_0^1 du \int d\oU \left(\nabla_{x_k}f({\xi}_{k-1})\right)^2 \hP_{\widetilde{\eta},\eta}\left(\eta^1_t\neq \eta^2_t\right) \label{eq:hardwork1b}\\
&\qquad\left.+ \int \mu(d\eta) \int_0^1 du \int d\oU \left(\nabla_{x}f(\xi)\right)^2 \hP_{\widetilde{\eta},\eta}\left(\eta^1_t\neq \eta^2_t\right) \label{eq:hardwork1c}\right].
\end{align}
Now we use H\"older's inequality with respect to the integration $ \int \mu(d\eta) \int_0^1 du \int d\oU$. In all three cases this produces as second term
\begin{align*}
&\left(\int \mu(d\eta) \int_0^1 du \int d\oU \ \hP_{\widetilde{\eta},\eta}\left(\eta^1_t\neq \eta^2_t\right)^q \right)^{\frac1q}. 	
\end{align*}
Note that, depending on $u$, $\tilde\eta$ is either $\eta^x$ or $\eta$, in which case $\hP_{\tilde\eta,\eta}(\eta^1_t\neq \eta^2_t)=0$. Using this as well as translation invariance shows that the above term equals
\[
\left(\int c(0,\eta) \theta_t(\eta)^q\,\mu(d\eta)\right)^{\frac1q}.
\]
The other term of H\"older's inequality varies slightly from line to line, but as it is mostly the same we focus on line \eqref{eq:hardwork1a}:
\begin{align*}
\left( \int \mu(d\eta) \int_0^1 du \int d\oU  \left(\nabla_{x_k}f(\widetilde{\xi}_{k-1})\right)^{2p} \right)^{\frac1p}
\end{align*}
Here we can finally use the fact that the configurations $\xi_k$,$\tilde{\xi}_k$ are $\mu$-distributed. Because of this fact we have the following identity:
\begin{align*}
&\left( \int \mu(d\eta) \int_0^1 du \int d\oU  \left(\nabla_{x_k}f(\widetilde{\xi}_{k-1})\right)^{2p} \right)^{\frac1p} \\
&\quad= \left( \int \mu(d\eta) \left(\nabla_{x_k}f(\eta)\right)^{2p} \right)^{\frac1p}
= \norm{(\nabla_{x_k}f)^2}_{L^p(\mu)}.
\end{align*}
Applying the same argument to \eqref{eq:hardwork1b} and \eqref{eq:hardwork1c},
\begin{align*}
&\sum_{x\in\bZ^d} \int \mu(d\eta) \int_0^1 du \left( S_t f(\Psi(\eta, (x,0,u))) - S_t f(\eta) \right)^2 \\
&\quad \leq \left(\int c(0,\eta) \theta_t(\eta)^q\,\mu(d\eta)\right)^{\frac1q} \\
&\qquad\qquad \sum_{x\in\bZ^d} \int d\oN \left(2\abs{C_{t,x}}+1\right) \left[ 2\suml_{k=1}^{\abs{C_{t,x}}}\norm{(\nabla_{x_k} f)^2}_{L^p(\mu)} + \norm{(\nabla_{x} f)^2}_{L^p(\mu)} \right].
\end{align*}
By translation invariance of the law of $\oN$, 
\begin{align*}
&\sum_{x\in\bZ^d} \int d\oN \left(2\abs{C_{t,x}}+1\right) \left[ 2\suml_{k=1}^{\abs{C_{t,x}}}\norm{(\nabla_{x_k} f)^2}_{L^p(\mu)} + \norm{(\nabla_{x} f)^2}_{L^p(\mu)} \right]	\\
&\quad= \sum_{x\in\bZ^d} \int d\oN \left(2\abs{C_{t,0}}+1\right) \left[ 2\suml_{k=1}^{\abs{C_{t,0}}}\norm{(\nabla_{x_k+x} f)^2}_{L^p(\mu)} + \norm{(\nabla_{x} f)^2}_{L^p(\mu)} \right]	\\
&\quad =  \int d\oN \left(2\abs{C_{t,0}}+1\right)^2 \sum_{x\in\bZ^d} \norm{(\nabla_{x} f)^2}_{L^p(\mu)} .
\end{align*}
\end{proof}
In order to proceed we need estimates on the size of $C_{t,0}$. The following two lemmas provides us with those.
\begin{lemma}\label{lemma:firstpassagepercolation}
Denote by $B_t\subset \bZ^d$ the set of sites which are represented in $C_{t,0}$, i.e.,
\[ B_t := \{ x\in\bZ^d \,|\, \exists\, s \in [0,t]:(x,s) \in C_{t,0} \} \cup \{0\}. \]
Then there exist dimension-dependent constants $c_1,c_2>0$ such that
\begin{enumerate}
\item $ \int \abs{B_t}^2\,d\oN \leq c_1 (t+1)^{2d} $; 
\item $ \sum_{x\in\bZ^d} \left(\int \ind_{x\in B_t}\,d\oN\right)^{\frac12} \leq c_2 (t+1)^d$. 
\end{enumerate}
\end{lemma}
\begin{proof}
\newcommand{\tB}{\widetilde{B}}
The proof rests on the observation that $B_t$ is strongly related to first passage percolation: 
Consider first passage percolation with iid. exponentially distributed edge weights(see for example \cite{KESTEN:86}). Let $E$ be the edge set of $\bZ^d$, and $r_e, e\in E$, independent and $\exp(1)$-distributed. Then the first passage percolation distance is $T(0,x) = \inf \{ \sum_{e\in \gamma} r_e \,|\, \text{$\gamma$ path from 0 to $x$} \}$. Now we compare the ball $\tB_t:=\{x\in \bZ^d : T(0,x)\leq t\}$ of reachable sites within distance $t$ to $B_t$ in terms of growth. Denote the outer boundary of a finite subset $A$ of $\bZ^d$ by $\partial A = \{x\in\bZ^d\bs A \,|\, \exists\,y\in A : \abs{x-y}=1 \}$. The rate at which a site $x\in\partial\widetilde{B}_t$ is encompassed by $\widetilde{B}_t$ is equal to the number of edges connecting $x$ to $\widetilde{B}_t$. On the other hand $B_t$ grows to contain a site $x\in\partial B_t$ just at rate 1. Therefore $\tB_t$ stochastically dominates $B_t$, and proving a) and b) for $\tB_t$ suffices.

From the theory of first passage percolation(see \cite{KESTEN:86}, Theorems 3.10, 3.11) we use the following fact :
There exist positive constants $k_1, k_2, k_3$ (possibly dimension-dependent) such that for all $x\in \bZ^d$ with $\abs{x}>k_1 t$:
\begin{align}
\bP(x\in \tB_t) =  \bP(T(0,x) \leq t) \leq k_2 e^{-k_3 \abs{x}}. \label{eq:fpp1}	
\end{align}
 To prove b),
\begin{align*}
\sum_{x\in\bZ^d} \bP(x\in \tB_t)^{\frac12} &\leq \sum_{x: \abs{x}\leq k_1 t} 1 +\sum_{x: \abs{x}> k_1 t}k_2e^{-k_3\abs{x} } \\
&\leq (2k_1+1)^d t^d + \sum_{x \in \bZ^d}k_2e^{-k_3\abs{x} } \\
&\leq c_2 (t+1)^d
\end{align*}
for a suitable constant $c_2$. To prove a), fix an integer $r> k_1 t$ and note that $\abs{\tB_t}> (2r+1)^d$ implies that at least one site in $\tB_t$ lies outside a cube of size $2r+1$. Combining this fact with estimate \eqref{eq:fpp1} gives us
\[ \bP\left(\abs{\tB_t}> (2r+1)^d\right) \leq \sum_{\norm{x}_\infty=r+1}\bP(x\in\tB_t) \leq k_2 e^{-k_3 (r+1)} 2d(2r+3)^{d-1}, \]
which proves exponentially decaying tails for the volume of $\tB_t$.
\end{proof}
Utilizing Lemma \ref{lemma:firstpassagepercolation} we now prove the second moment estimate of $\abs{C_{t,0}}$ needed for Lemma \ref{lemma:hardwork}.
\begin{lemma}\label{lemma:clusterestimate}
There exists a dimension-dependent constant $C_d>$ so that the following estimate holds:
\[ \int (2\abs{C_{t,0}}+1)^2\,d\oN \leq C_d (t+1)^{2d+2}. \]
\end{lemma}
\begin{proof}
Let $B_t$ be as in Lemma \ref{lemma:firstpassagepercolation}. 
Then for each $x\in B_t$ we denote by $t_x$ the time of first time of appearance of $x$ in $C_{t,0}$,
\[ t_x:= \inf\{s\in[0,t] \,|\,(x,s)\in C_{t,0}\}. \]
We have
\[ C_{t,0} = \oN \cap \{ (x,s)\in\bZ^d\times [0,t] \,|\, x\in B_t, s\geq t_x \} \subset \oN \cap \{ (x,s)\in\bZ^d\times [0,t] \,|\, x\in B_t \}. \]
Conditioned on $B_t$ and $t_x$ the last set is Poisson distributed with the addition of the points $(x,t_x),x\in B_t$. Because of this, conditioned on $B_t$, $\abs{C_{t,0}}-\abs{B_t}$ is stochastically dominated by a Poisson distributed with parameter $t\abs{B_t}$. As a consequence, 
\[ \int (2\abs{C_{t,0}}+1)^2\,d\oN \leq 4 \int (t+1)^2(\abs{B_t}+1)^2\,d\oN. \]
Finally the estimate from Lemma \ref{lemma:firstpassagepercolation},a) completes the proof.
\end{proof}
With all ingredients present we can quickly prove the main result in form of a slightly more general lemma.
\begin{lemma}\label{lemma:main-precursor}
Let $f:\Omega\to\bR$ with $\normb{f}_2<\infty$ and $0\leq T\leq S$. Then
\begin{align*}
 &\Var_\mu(S_T f) - \Var_{\mu}(S_S f) \\
 &\quad\leq C_d \int_T^S (t+1)^{2d+2} \left(\int c(0,\eta) \theta_t(\eta)^q\,\mu(d\eta)\right)^{\frac1q} \,dt \sum_{x\in\bZ^d} \norm{(\nabla_x f)^2}_{L^p(\mu)}.  
 \end{align*}
$C_d$ is a constant depending just on the dimension.
\end{lemma}
\begin{proof}
Assume that $f$ satisfies $\normb{f}_1<\infty$. This then implies that $S_t f, (S_tf)^2 \in dom(L)$ and by Lemma \ref{lemma:variance}, 
\[ \Var_\mu(S_T f)-\Var_\mu(S_S f) = \int_T^S \int \left[L(S_tf-S_t f(\eta))^2\right](\eta)\,\mu(d\eta)\,dt .\]
By using the formulation of the generator using the graphical construction (see \eqref{eq:generator-graphical-representation}),
\[ Lf(\eta) = \sum_{x\in\bZ^d}\int_0^1 \left[f\left(\Psi(\eta,(x,0,u))\right) - f(\eta)\right]\,du, \]
we apply Lemma \ref{lemma:hardwork} and obtain
\begin{align*}
\Var_\mu(S_T f)- \Var_\mu(S_S f) \leq \int_T^S \int (2\abs{C_{t,0}}+1)^2\,d\oN \norm{\theta_t}_{L^q(\mu)} \,dt \sum_{x\in\bZ^d} \norm{(\nabla_x f)^2}_{L^p(\mu)}.
\end{align*}

Finally Lemma \ref{lemma:clusterestimate} gives us the estimate on $\int (2\abs{C_{t,0}}+1)^2\,d\oN$ to complete the proof. 

If $f$ only satisfies $\normb{f}_2<\infty$ we then approximate $f$ by local functions.
\end{proof}
\begin{proof}[Proof of Theorem \ref{thm:main-result}]
A direct consequence of Lemma \ref{lemma:main-precursor} with $S=\infty$ and the estimate 
$\left(\int c(0,\eta) \theta_t(\eta)^q\,\mu(d\eta)\right)^{\frac1q} \leq \norm{\theta_t}_{L^q(\mu)} $.
\end{proof}

We now prove Theorem \ref{thm:auto-correlation}, which is a modification of Theorem \ref{thm:main-result} for attractive spin-systems. 
\begin{proof}[Proof of Theorem \ref{thm:auto-correlation}]
This result is also based on Lemma \ref{lemma:main-precursor}. To estimate $\left(\int c(0,\eta) \theta_t(\eta)^q\,\mu(d\eta)\right)^{\frac1q}$ in terms of the auto-correlation, we start with the fact that in the coupling the spread of discrepancies is limited to $B_t$(as in Lemma \ref{lemma:firstpassagepercolation}):
\begin{align*}
\theta_t(\eta) &= \hP_{\eta^0,\eta}(\eta^1_t\neq\eta^2_t) \leq \hE_{\eta^0,\eta}\sum_{x\in B_t} \ind_{\eta^1_t(x)\neq\eta^2_t(x)} = \sum_{x\in\bZ^d} \hE_{\eta^0,\eta} \ind_{x\in B_t} \ind_{\eta^1_t(x)\neq\eta^2_t(x)} .
\end{align*}
Next, since $\theta_t\leq 1$,
\begin{align*}
\int c(0,\eta) \theta_t(\eta)^q \,\mu(d\eta) 
&\leq \int c(0,\eta) \theta_t(\eta) \,\mu(d\eta) \\
&\leq  \sum_{x\in\bZ^d} \int \hE_{\eta^0,\eta} \ind_{x\in B_t} \ind_{\eta^1_t(x)\neq\eta^2_t(x)}c(0,\eta) \,\mu(d\eta). 	
\end{align*}
We can now use Cauchy-Schwarz to obtain
\begin{align}\label{eq:ising1}
&\sum_{x\in\bZ^d} \left( \int \ind_{x\in B_t} \,d\oN \right)^{\frac12} \left( \int \hE_{\eta^0,\eta} \ind_{\eta^1_t(x)\neq\eta^2_t(x)} c(0,\eta)^2\,\mu(d\eta) \right)^{\frac12}.	
\end{align}
Since the model is attractive the coupling $\hP$ preserves an initial ordering. Since either $\eta^0 < \eta$ or $\eta < \eta^0$,
\begin{align*}
\hE_{\eta^0,\eta} \ind_{\eta^1_t(x)\neq\eta^2_t(x)} &= \frac12\hE_{\eta^0,\eta} \abs{\eta^1_t(x) - \eta^2_t(x)} = \frac12\abs{\hE_{\eta^0,\eta} \left(\eta^1_t(x)-\eta^2_t(x)\right)} \\
&= \frac12\abs{\bE_{\eta^0} \eta_t(x)- \bE_\eta \eta_t(x)}	.
\end{align*}
When we use the notation $g_x(\eta) := \eta(x)$ and $m=\mu(g_0)=\mu(g_x)$, 
\begin{align*}
&\int \hE_{\eta^0,\eta} \ind_{\eta^1_t(x)\neq\eta^2_t(x)} c(0,\eta)^2\,\mu(d\eta) 
= \frac12\int \abs{S_t g_x(\eta^0)- S_t g_x(\eta)} c(0,\eta)^2\,\mu(d\eta) 	\\
&\quad\leq \frac12\int \abs{ S_t g_x(\eta^0) -m} c(0,\eta)^2\,\mu(d\eta) + \frac12\int \abs{S_t g_x(\eta)-m } c(0,\eta)^2 \,\mu(d\eta) .
\end{align*}
We have
\begin{align*}
\int\abs{S_t g_x(\eta)-m } \,\mu(d\eta) &\leq \left(\int\left(S_t g_x(\eta)-m \right)^2\,\mu(d\eta)\right)^{\frac12} \\
& = \Var_\mu(S_t g_x)^{\frac12} = \Var_\mu(S_t g_0)^{\frac12}
\end{align*}
and by using reversibility and $c(0,\eta^0)\leq1$ we also have
\begin{align*}
\int \abs{ S_t g_x(\eta^0) -m }c(0,\eta) \,\mu(d\eta)
&= \int \abs{ S_t g_x(\eta^0) -m }c(0,\eta) \frac{\mu(d\eta)}{\mu(d\eta^0)} \,\mu(d\eta^0)	\\
&\leq \int \abs{ S_t g_x(\eta^0) -m } \,\mu(d\eta^0)	\leq \Var_\mu(S_t g_0)^{\frac12}.
\end{align*}
Combining the estimates we obtain
\[ \int c(0,\eta) \theta_t(\eta)^q \,\mu(d\eta) \leq  \sum_{x\in\bZ^d} \left( \int \ind_{x\in B_t} \,d\oN  \right)^{\frac12} \Var_\mu(S_tg_0) ^{\frac14} . \]
Furthermore Lemma \ref{lemma:firstpassagepercolation} gives us an estimate for the sum, so that
\[ \left( \int c(0,\eta) \theta_t(\eta)^q \,\mu(d\eta) \right)^{\frac1q} \leq c_2^{\frac1q}(t+1)^{\frac{d}{q}}\Var_\mu(S_tg_0) ^{\frac1{4q}}.  \]
Omitting the $q$-th root where convenient and with a constant $C_d' = C_d (1\vee c_2)$ as well as writing $\frac1q = \frac{p-1}{p}$ we obtain the result 
\[ D_p(T)\leq C_d'\int_T^\infty (t+1)^{3d+2} \left(\Var_\mu(S_t g_0)\right)^{\frac{p-1}{4p}}\,dt. \]
\end{proof}
\newcommand{\smu}[1][f]{E_{#1,#1}}
\begin{proof}[Proof of Proposition \ref{prop:spectral-gap}]
Let $f:\Omega\to\bR$ be a local function with $\mu(f)=0$. By Theorem \ref{thm:main-result}, for any $0<\lambda'<\lambda $,
\[ \norm{S_t f}^2_{L^2(\mu)} \leq const\cdot e^{-\lambda' t}, \]
and for any $0<a<\lambda'$
\begin{align}\label{eq:resolvent} 
\int_0^\infty e^{at}\norm{S_t f}^2_{L^2(\mu)}\,dt < \infty.
\end{align}
Let $\smu$ be the associated measure wrt. to the spectral decomposition of $-L$. Then
\[ \norm{S_t f}^2_{L^2(\mu)} = \int_0^\infty e^{-2\gamma t} \smu(d\gamma). \]
By \eqref{eq:resolvent}, 
\[ \int_0^\infty \int_0^\infty e^{at-2\gamma t} \smu(d\gamma) \,dt <\infty. \]
Therefore $\smu(]0,\lambda/2[)=0$ for any local function $f$.

Let now $f\in L^2(\mu)$ and approximate it by local functions $f_n$. Assuming that $(f_n), f$ have norm 1 makes $\smu[f_n],\smu[f]$ probability measures and $\smu[f_n]$ weakly converges to $\smu[f]$. By the Portmanteau theorem $\smu[f](]0,\lambda/2[)=0$, which completes the proof.
\end{proof}
\begin{proof}[Proof of Proposition \ref{prop:reverse-spectral-gap}]
By the Poincar\'e inequality, the auto-correlation of the spin at the origin, $\phi(t)=\Var_\mu(S_t g), g(\eta)=\eta(0),$ decays exponentially fast. The proof of Theorem \ref{thm:auto-correlation} contains the estimate of $\left(\int c(\eta,0)\theta_t(\eta)^q \,\mu(d\eta)\right)^{\frac1q}$ in terms of $\phi$.
\end{proof}
\begin{proof}[Proof of Proposition \ref{prop:L-infty-decay}]
We have 
\[ \norm{S_Tf-\mu(f)}_{\infty} = \norm{\int_T^\infty LS_t f\,dt} \leq \sup_{\eta\in\Omega}\int_T^\infty \sum_{x\in\Omega}\abs{\nabla_x S_t f(\eta)} dt. \]
Write $\delta_x(f) := \norm{\nabla_xf}_\infty$. Then
\begin{align*}
&\sum_{x\in\bZ^d} \abs{\nabla_xS_tf(\eta)}
\leq \sum_{x\in\bZ^d} \hE_{\eta^x,\eta}\abs{f(\eta^1_t)-f(\eta^2_t)}	\\
&\leq \sum_{x\in\bZ^d} \hE_{\eta^x,\eta}\left(\sum_{y\in\bZ^d}\ind_{\eta^1_t(y)\neq\eta^2_t(y)}\delta_y(f)\right)
= \hE_{\eta^0,\eta}\sum_{y\in\bZ^d}\ind_{\eta^1_t(y)\neq\eta^2_t(y)} \normb{f}_1 	\\
&\leq \hE_{\eta^0,\eta}\abs{C_{t,0}}\ind_{\eta^1_t\neq \eta^2_t} \normb{f}_1\\
\end{align*}
By the Cauchy-Schwarz inequality and Lemma \ref{lemma:clusterestimate},
\begin{align*}
\hE_{\eta^0,\eta}\abs{C_{t,0}}\ind_{\eta^1_t\neq \eta^2_t} \leq \hP_{\eta^0,\eta}(\eta^1_t \neq \eta^2_t)^{\frac12}\left(\hE_{\eta^0,\eta}\abs{C_{t,0}}^2\right)^{\frac12}
\leq \norm{\theta_t}_{L^\infty(\mu)}^{\frac12} C_d^{\frac12} (t+1)^{d+1}.
\end{align*}
Therefore
\[ \norm{S_T-\mu(f)}_{\infty} \leq \int_T^{\infty} C_d^{\frac{1}{2}} (t+1)^{d+1} \norm{\theta_t}_{L^\infty(\mu)}^{\frac{1}{2}} \,dt \normb{f}_1.  \qedhere\]
\end{proof}
\begin{proof}[Proof of Proposition \ref{prop:weak-poincare}]
Fix $R>0$. Then, by Lemma \ref{lemma:main-precursor},
\begin{align*}
\Var_\mu(S_t f) &\leq C_d D_1(0,R) \sum_{x\in\bZ^d} \norm{(\nabla_x S_t f)^2}_{L^1(\mu)} + \Var_\mu(S_R S_t f) \\
&\leq C_d \delta^{-1} D_1(0,R) \cE(S_tf,S_tf) + \Var_\mu(S_R f)\\
&\leq C_d \delta^{-1} D_1(0,R) \cE(f, f) + D_p(R) \Phi_R(f).
\end{align*}
The estimate $\Phi_R(S_t f) \leq C_d \sum_{x\in\bZ^d} \norm{(\nabla_x S_t f)^2}_{L^p(\mu)}$ is a direct consequence of Theorem \ref{thm:main-result}. 
\end{proof}
\begin{proof}[Proof of Proposition \ref{prop:uniform-condition}]
Set $d(\eta,\xi):=\sum_{x\in\bZ^d}\ind_{\eta(x)\neq\xi(x)}$ and $h(t):=\sup_{\eta\in\Omega}\hE_{\eta^0,\eta}d(\eta^1_t,\eta^2_t)$.
Then $\norm{\theta_t}_\infty\leq h(t)$. We observe that
\begin{align*}
 h(t+s)&=\sup_{\eta\in\Omega}\hE_{\eta^0,\eta}\hE_{\eta^1_t,\eta^2_t}d(\eta^1_s,\eta^2_s) \\
 &\leq \sup_{\eta\in\Omega}\hE_{\eta^0,\eta}d(\eta^1_t,\eta^2_t)h(s) = h(t)h(s).
\end{align*}
The inequality follows from the fact that $\hE_{\eta,\xi}d(\eta^1_t,\eta^2_t)\leq \hE_{\eta,\zeta}d(\eta^1_t,\eta^2_t) + \hE_{\zeta,\xi}d(\eta^1_t,\eta^2_t)$ applied to an interpolating sequence between $\eta^1_t$ and $\eta^2_t$. A quick calculation shows that
\[ \lim_{t\to0}\frac1t (h(t)-h(0)) = \sup_{\eta\in\Omega}\sum_{\abs{x}=1}\abs{c(x,\eta^0)-c(x,\eta)} -1. \]
Therefore
\[ h(t)\leq \exp\left( \left(\sup_{\eta\in\Omega}\sum_{\abs{x}=1}\abs{c(x,\eta^0)-c(x,\eta)} -1\right)t \right). \qedhere\]
\end{proof}

\begin{proof}[Proof of Proposition \ref{prop:DSM}]
We have 
\[ \theta_t(\eta) = \hP_{\eta^0,\eta}(\eta^1_t\neq \eta^2_t) \leq \sum_{x\in\bZ^d} \hP_{\eta^0,\eta}(\eta^1_t(x) \neq \eta^2_t(x)) \]
By the same comparison argument with first passage percolation as in Lemma \ref{lemma:firstpassagepercolation}, there are constants $k_1,k_2,k_3>0$ so that for $\abs{x}\geq k_1 t $ we have
\begin{align}\label{eq:large-x}
 \hP_{\eta^0,\eta}(\eta^1_t(x) \neq \eta^2_t(x)) \leq k_2 e^{-k_3 \abs{x}}. 
 \end{align}

Now we look at $\abs{x}<k_1 t$.
Write $f_x(\eta) := \eta(x)$. By the assumption of attractiveness,
\[ \hP_{\eta^0,\eta}(\eta^1_t(x) \neq \eta^2_t(x)) = \abs{S_t f_x(\eta^0) - S_t f_x(\eta)}. \]
We denote by $\Lambda_L$ the cube of side length $L$ and $S_t^{\Lambda_L,\xi}$ the finite volume semi-group with boundary condition $\xi$. Since $f_x$ is local, we have 
\[ \lim_{L\to\infty} \norm{S_t^{\Lambda_L, \xi}f_x - S_t f_x}_\infty = 0. \]
Therewith as well as \eqref{eq:DSM-Linfty},
\begin{align}\label{eq:small-x}
\sup_{\eta\in\Omega}\abs{S_t f_x(\eta^0) - S_t f_x(\eta)} &\leq 2 \sup_{L>0,\xi\in\Omega} \norm{S_t^{\Lambda_L, \xi}f_x - \mu^{\Lambda_L, \xi}}_\infty \leq Ce^{-\lambda t}.
\end{align}
Combining \eqref{eq:large-x} and \eqref{eq:small-x} we obtain exponential decay of $\theta_t$ and thereby $D_1$.
\end{proof}

\section*{Acknowledgement}
We thank NWO for financial support, under the project `vrije competitie' number 600.065.100.07N14.

\bibliography{BibCollection}{}
\bibliographystyle{plain}

\end{document}